\documentclass[11pt]{article}
\usepackage{amssymb, latexsym, mathtools, amsthm}
\begingroup\makeatletter
\@for\theoremstyle:=definition,remark,plain\do{%
\expandafter\g@addto@macro\csname th@\theoremstyle\endcsname{%
\addtolength\thm@preskip\parskip}}
\endgroup
\usepackage{authblk}
\usepackage{stmaryrd}
\usepackage{enumerate}
\usepackage{pgf,tikz}
\usetikzlibrary{decorations, decorations.pathmorphing}
\usetikzlibrary {shapes}
\usepackage{pdfsync}
\usepackage{setspace}
\usepackage[round]{natbib}
\usepackage{txfonts}
\usepackage[T1]{fontenc}
\usepackage{booktabs}
 \usepackage{graphicx}
\definecolor{dnrbl}{rgb}{0,0,0.3}
\definecolor{dnrgr}{rgb}{0,0.3,0}
\definecolor{dnrre}{rgb}{0.5,0,0}
\usepackage[colorlinks=true, citecolor=dnrgr, linkcolor=dnrre, urlcolor=dnrre] {hyperref}
\usepackage{xcolor}
\usepackage{parskip}
\usepackage{tabularx, colortbl}
\usepackage{geometry}
 \renewenvironment{abstract}
 { \normalsize
  \list{}{\setlength{\leftmargin}{.0cm}%
    \setlength{\rightmargin}{\leftmargin}}%
  \item {\bf \abstractname.} \relax}
 {\endlist}
\theoremstyle{plain}
\newtheorem{thm}{Theorem}[section]

\newtheorem{lem}[thm]{Lemma}
\newtheorem{coro}[thm]{Corollary}

\theoremstyle{definition}

\newtheorem{defi}[thm]{Definition}
\makeatletter

\let\c@table\c@figure
\makeatother

\newcommand{\Nat}{\mathbb{N}}

\newcommand{\restr}{\upharpoonright}  
 
\DeclarePairedDelimiter{\ceil}{\lceil}{\rceil}

\makeatletter
\newcommand{\oset}[3][0ex]{%
  \mathrel{\mathop{#3}\limits^{
    \vbox to#1{\kern-2\ex@
    \hbox{$\scriptstyle#2$}\vss}}}}
\makeatother
\newcommand{\abs}[1]{|\hspace{0.03cm}{#1}\hspace{0.03cm}|}

\newcommand{\zp}{\mathbf{0}'}
\newcommand{\ze}{\mathbf{0}}

\newcommand{\parb}[1]{\big(\hspace{0.03cm}{#1}\hspace{0.03cm}\big)}

\newcommand{\parB}[1]{\Big(\hspace{0.1cm}{#1}\hspace{0.1cm}\Big)}

\newcommand{\sqbrad}[2]{\big\{\hspace{0.03cm}{#1}\hspace{0.03cm}:\hspace{0.03cm}{#2}\hspace{0.03cm}\big\}}

\DeclarePairedDelimiter{\dbra}{\llbracket}{\rrbracket}

\newcommand{\BB}{\mathcal{B}}

\newcommand{\wedga}{\ \wedge\ \ }

\newcommand{\ml}{Martin-L\"{o}f }

\newcommand{\KG}{Ku\v{c}era-G{\'a}cs\ }
\newcommand{\pz}{$\Pi^0_1$\ }

\newcommand{\eg}{e.g.\ }

\newcommand{\ce}{c.e.\ }

\newcommand{\pf}{prefix-free }

\newcommand{\twome}{2^{\omega}} 
\newcommand{\twomel}{2^{<\omega}}

\newcommand{\PA}{\mathsf{PA}}
\newcommand{\DOM}{\mathsf{DOM}} 
\newcommand{\DNC}{\mathsf{DNC}}

\newcommand{\bigon}{\textrm{\textup O}\hspace{0.02cm}(1)}

\usepackage{datetime}

\title{Randomness  below complete theories of arithmetic\thanks{Supported by NSFC grant No.~11971501 and partially motivated by the participation  in the virtual program of the Institute for Mathematical Sciences, National University of Singapore, in 2021.}}
\author[1]{George Barmpalias}
\affil[1]{State Key Lab of Computer Science, Inst.\ of Software, Chinese Acad.\ of Sciences, Beijing, China\vspace{0.1cm}}
\author[2]{Wei Wang}
\affil[2]{Inst.\ of Logic \& Cognition and Dept.\ of Philosophy, Sun Yat-Sen University, Guangzhou, China}

\begin{document}
\maketitle
\begin{abstract}
We show that reals $z$ which compute complete extensions of arithmetic
have the random join property: for each random $x<_T z$ 
there exists random $y<_T z$ such that $z\equiv_T x\oplus y$. 
The same is true for the truth-table and the weak truth-table reducibilities.
\end{abstract}
\setcounter{tocdepth}{1}
\tableofcontents

\section{Introduction}\label{j19EOHHQvH}
The incompleteness phenomenon in  arithmetic, discovered by \citet{godel31}, implies that there is no computable binary predicate 
consistently evaluating the truth of each arithmetical sentence. 
Formal arithmetic is known as Peano arithmetic as it is generated by Peano's axioms. Theories of arithmetic can be identified with
the infinite binary sequences, {\em reals}, under a standard fixed coding of formulas into positive integers. 
Let $\PA$ denote the set of reals that encode 
complete extensions of arithmetic; the set $\PA$ forms an effectively closed set, a \pz class.
The computational power of members of $\PA$, in terms of their Turing degrees, has been investigated  since 
\citep{ScottTennen}, and its systematic study started by \citet{MR0316227}.

\citet{Levin2013FI} pointed out that the relationships
between $\PA$ and algorithmically random reals,  in the sense of \citet{MR0223179}, 
are essential for understanding the limitations of obtaining solutions to arithmetical problems by probabilistic methods.
These relationships have been under investigation since the work of \citet{MR820784}.

We show that the $\PA$ degrees (Turing degrees of reals in $\PA$) 
have the {\em  join property} with respect to random degrees (Turing degrees of random reals).
In general, a degree $\mathbf{a}$ is said to have the
{\em join property} if, for every non-zero $\mathbf{b}<\mathbf{a}$, there exists 
$\mathbf{c}<\mathbf{a}$ with $\mathbf{b}\vee \mathbf{c}=\mathbf{a}$. 

The $\PA$ degrees do not, in general, satisfy the join property \citep{lewisjoinnote}.
In contrast, we show that for every $\PA$ degree $\mathbf{a}$, random 
$\mathbf{b}<\mathbf{a}$ there exists random
$\mathbf{c}<\mathbf{a}$ with $\mathbf{b}\vee \mathbf{c}=\mathbf{a}$. 
\begin{thm}\label{C1JiRsNKu}
If $z$ is $\PA$ and $x<_T z$ is random, there exists random $y<_T z$ such that $x\oplus y \equiv_T z$.
The same is true for truth-table and weak truth-table reducibility:  $<_{tt}, \equiv_{tt}$ and $<_{wtt}, \equiv_{wtt}$.
\end{thm}
The universality of $\PA$ reals implies that they compute some random real. Hence Theorem \ref{C1JiRsNKu}
generalizes an older result by \citet{minpal}, which says that for each $\PA$ real $z$ there exist random $x,y$
such that $z\equiv_T x\oplus y$. The latter was used by \citet{moserdepth} to derive an interesting fact about
{\em logical depth}, invented by \citet{Bennett1988}
in order to quantify the hardness of obtaining useful information via probabilistic algorithms.

Intuitively, a real is {\em deep} if it cannot be produced by a randomized algorithm with non-trivial probability. The halting problem, 
for example, is deep.
Reals that are not deep are called {\em shallow}, and include the computable and the random reals.
Depth is not invariant with respect to Turing equivalence, but it is invariant with respect to computations with computable time-bound;
equivalently, deep reals are upward closed with respect to truth-table reductions.

\citet{moserdepth} showed that there are shallow $x,y$ such that $x\oplus y$ is deep. We obtain the following extension, using 
Theorem \ref{C1JiRsNKu}. A real $x$ is {\em $\ze$-dominated} 
(also known as {\em hyperimmune-free}) if every $x$-computable function is dominated by a computable function.
\begin{coro}\label{bZmBhYeMn}
Let $\DOM$ denote the class of $\ze$-dominated reals.
\begin{enumerate}[\hspace{0.3cm}(a)]
\item For each  random $x\in\DOM$,  there exists random $y\in\DOM$ such that $x\oplus y\in\DOM\cap \PA$.
\item For each random $x$, there exists random $y$ such that $x\oplus y$ is deep.
\end{enumerate}
\end{coro}
The state-of-the-art regarding random and $\PA$ reals is given in \S\ref{UxUy99Ez}, and the proofs of our results are
presented in \S\ref{pF9WGGxG8R}. We end with a brief discussion in \S\ref{YzKdhxVSuD}, which includes 
questions and suggestions on the same topic. 

\section{Background and state-of-the-art}\label{UxUy99Ez}
For an introductory survey on  $\PA$, random and $\ze$-dominated reals and degrees, look no further than \citep{DzhafarovPA}.
Let $\twome$ denote the set of infinite binary sequences and let $\twomel$ denote the set of binary strings. We also use $2^n$ to denote the set of
$n$-bit strings and $2^{\geq n}$ for the set of strings of length $\geq n$.

We review the aspects of these topics that are directly relevant to our result.

\subsection{Random and PA reals}\label{iDz5who56}
A real in $\PA$ need not compute the halting set,
and indeed can have  low degree of unsolvability.
Viewed as an infinite binary sequence, a $\PA$  {\em real} $z$ encodes a selection with
respect to every nonempty \pz subset of $\Nat$:
\begin{equation}\label{zEnN2ZYwr6}
\parbox{12cm}{there exists total $f\leq_{tt} z$ such that if $e$ is an effective description of a \pz set $D_e$ 
(as the set of solutions of a \pz predicate) either $D_e=\emptyset$ or $f(e)\in D_e$.}
\end{equation}
In other words, $z$ may not be able to tell if $D_e$ is empty, but it can choose a number which will be in $D_e$ unless $D_e$ is empty.
Similarly, a $\PA$ real computes a member in any given nonempty \pz subset of $\twome$, uniformly on the index of the class.

Another characterization is in terms of {\em diagonally noncomputable $\DNC$ functions}, namely functions $f$ with $\forall n,\ f(n)\not\simeq \varphi_e(e)$,
where $\varphi_e$ is a computable enumeration of all partial computable functions.
A Turing degree is $\PA$ iff it contains a boolean $\DNC$ function. In \S\ref{ChCj3ZFGWn} we will discuss {\em logical-depth} which is invariant up to
truth-table degrees, but not up to Turing degrees. It is therefore important to distinguish  {\em $\PA$ reals} from reals of {\em $\PA$ degree}.
One can think a $\PA$ real as a $\DNC$ function with binary range, or a real that effectively encodes choice-functions for \pz sets, in a truth-table way
(equivalently, the retrieval of information occurs within a fixed computable time-bound). 

A subclass of $\DNC$ degrees (Turing degrees that contain a $\DNC$ function) is the class of \ml random, also known as 1-random (or simply {\em random}), degrees. 
Let $\mu$ denote the standard Lebesgue measure on $\twome$. A uniformly \ce sequence $U_i\subseteq\twomel$, viewed as effectively open sets (namely the set $\dbra{U_i}$ of reals which have a prefix in $U_i$), with $\mu(U_i)\leq 2^{-i}$ is called a {\em \ml test}. A real $x$ is {\em random} if for each \ml test $(U_i)$ only finitely many members $U_i$ contain a prefix of $x$.
There exists a {\em universal \ml test}: a test $(U_i)$ such that for every $i$ and every \ml test $(V_j)$,  we have $\dbra{V_j}\subseteq \dbra{U_i}$ for all but finitely many $j$. Hence $\twome-\dbra{U_i}$ consists entirely of random reals, and 
the random reals are exactly the union of these sets with respect to $i$. 
These notions can be relativized with respect to any real $z$, where the test $(U_i)$ is only required to be \ce relative to $z$.
Randomness relative to the halting problem $\zp$ is called {\em 2-randomness}.

The complement of a member $U_i$ of a universal test can be viewed as a \pz class of positive measure,
consisting entirely of random reals. By the properties discussed above, every real of $\PA$ degree computes a random real. 

Further results connect  these classes of reals and their degrees:
\begin{enumerate}[\hspace{0.3cm}(i)]
\item every random real computes a $\DNC$ function \citep{MR820784}
\item every $\PA$ degree is the join of two random degrees \citep{minpal}
\item a random $x$ has $\PA$ degree iff $x\geq_T \zp$ \citep{MR2258713frank}
\end{enumerate}
and in (i), (ii) the reductions are truth-table.

\subsection{Kolmogorov complexity, PA and logical depth}\label{ChCj3ZFGWn}
The length of the shortest program for the universal Turing machine that outputs $\sigma$ is known as the
{\em Kolmogorov complexity} of $\sigma$, and is denoted by $C(\sigma)$. The same statement in terms of self-delimiting or \pf machines\footnote{Turing machines
with domain a \pf subset of $\twomel$.} defines the  {\em \pf Kolmogorov complexity} of $\sigma$,  denoted by $K(\sigma)$; it is a refined
information content measure  compared to $C(\sigma)$.
The \ml reals defined in \S\ref{iDz5who56} are exactly the reals $x$ whose prefixes cannot be compressed by more than a constant:
$\exists c\ \forall n\ K(x\restr_n)\geq n-c$.

The classes of $\PA$, $\DNC$ and Turing-hard (namely computing $\zp$) reals can be characterized in terms of Kolmogorov complexity 
\citep[\S 4]{KjosHanssenMStrans}:
\begin{itemize}
\item $x$ computes a $\DNC$ function iff $\exists f\leq_{T} x\ \forall n: C(f(n))\geq n$ 
\item $x$ truth-table computes a $\DNC$ function iff $\exists f\leq_{tt} x\ \forall n: C(f(n))\geq n$ 
\item $x$ computes a real in $\PA$ iff  $\exists f\leq_{tt} x\ \forall n, \sigma\in 2^n\ \parb{f(n)\in 2^n\wedga C(f(n))\geq C(\sigma)}$
\item $x\geq \zp \iff \exists f\leq_T x\ \forall n, \sigma\in 2^n\ \parb{f(n)\in 2^n\wedga K(f(n))\geq K(\sigma)}$
\end{itemize}
\citet{pamiller, relafrankjoh} showed that computing a $\PA$ real is equivalent to computing a  martingale  that 
majorizes the optimal c.e.\ supermartingale. These examples demonstrate the ubiquity of $\PA$ degrees in
computability and algorithmic information. 
 
Many of the above results about $\PA$ and random reals indicate a 
qualitative difference between the information
in an incomplete random real and the information in a $\PA$ real. The latter is highly structured and useful, as it allows to obtain a completion of
any partial computable predicate. In contrast, although the information content in random reals  is very large in terms of their Kolmogorov complexity,
this information appears unusable to a large degree.
This type of qualification of information was formalized by \citet{Bennett1988} through the notion of {\em logical-depth} which applies to strings and reals.
Several variants have been studied since, but as shown by \citet{MOSER201396},
all depth notions can be interpreted in the compression framework: 
\begin{equation*}
\parbox{11cm}{a sequence is {\em deep} if given  unlimited time, 
a compressor can compress the sequence $r(n)$ more bits than if given at most $t(n)$ time steps.}
\end{equation*}
By considering different time-bounds $t(n)$ and {\em depth-magnitudes} $r(n)$, all existing depth notions can be expressed in the compression framework.
\citet{Bennett1988} defined depth in terms of \pf Kolmogorov complexity, and one of his simplest formulations is
obtained by considering computable $t(n)$ and $r(n) = \bigon$. 
This was extended  by \citet{moserdepth}, and we use part of their terminology in the following
definition due to \citet{Bennett1988}.

A nondecreasing unbounded $g:\Nat\to\Nat$ is called an {\em order}.
Let $K^t$ denote the restriction of $K$ with respect to time-bound $t$.\footnote{Let $U_t$ 
denote the computations by the universal \pf machine up to step $t(n)$ for inputs of length $n$, and let
$K^t$ denote the \pf Kolmogorov complexity relative to $U_t$.}
\begin{defi}
A real $x$ is {\em $\bigon$-deep} if
for every computable time-bound $t$, every $c$ and almost all $n$, $K^t(x\restr_n)>K(x\restr_n)+c$.
A real $x$ is {\em order-deep} if there exists a computable order $g$ such that
for every computable time-bound $t$,  and almost all $n$, $K^t(x\restr_n)>K(x\restr_n)+g(n)$.
\end{defi}
Order-deep implies $\bigon$-deep.
The facts we present hold for both classes, so we simply talk of {\em deep reals} to refer to either notion.
A real that is not deep is called {\em shallow}.
Intuitively, deep reals are compressible but any computably time-bounded process can only produce a far from optimal compression. 

\citet{Bennett1988} and \citet{moserdepth} observed:
\begin{enumerate}[(i)]
\item computable and random reals are shallow 
\item the halting set and, in fact every $\PA$ real,  is deep 
\item if $x$ is deep then each $y\geq_{tt} x$ is deep
\item there are shallow $x,y$ such that  $x\oplus y$ is deep.
\end{enumerate}
Item (iv) involves finding random $x,y$ such that 
$x\oplus y$ is truth-table equivalent to a $\PA$ real.

\subsection{Join properties of PA and random degrees}\label{pY7jSdtn81}
Our result is a join property of $\PA$ degrees with respect to random degrees.
The {\em join property} is
one of the central algebraic properties in the study of the structure of the Turing degrees:
\begin{defi}
 A Turing degree $\mathbf{a}$ satisfies the 
\begin{itemize}
\item  {\em join property} if, for every non-zero $\mathbf{b}<\mathbf{a}$, there exists 
 $\mathbf{c}<\mathbf{a}$ with $\mathbf{b}\vee \mathbf{c}=\mathbf{a}$. 
\item {\em cupping property} if, for all $\mathbf{b}>\mathbf{a}$, there exists 
 $\mathbf{c}<\mathbf{b}$ with $\mathbf{a}\vee \mathbf{c}=\mathbf{b}$. 
\end{itemize}
\end{defi}
Does every $\PA$ degree have the join property? That was a question raised by \citet{MR820784}
and answered in the negative by
\citet{lewisjoinnote}.\footnote{\citet{MR820784} calls the degrees having the join property {\em cuppable}.}
Recall that $x$ is {\em low} if $x'\equiv_T\emptyset'$, 
namely the halting problem relative to $x$ is Turing-equivalent to the unrelativized halting set $\emptyset'$.
 The
\begin{itemize}
\item low $\DNC$ and low random degrees do not have the join property \citep{lewisjoinnote}
\item low $\PA$ degrees do not have the join property \citep{lewisjoinnote}
\item  2-random degrees have the join property  \citep{typical}
\item $\PA$ degrees have the cupping property \citep{MR820784}
\item  2-random degrees do not have the cupping property  \citep{typical}.
\end{itemize}
Theorem \ref{C1JiRsNKu} cannot be extended with respect to other 
standard algebraic properties studied in the Turing degrees, such as bounding a minimal pair, being the supremum of a minimal pair, 
or even having the meet or complementation property.\footnote{An informative table of these conditions can be found in \citep{typvol}.}  
It is known that:
\begin{enumerate}[\hspace{0.3cm}(i)]
\item any pair of $\DNC$ degrees below $\zp$ fails to be a minimal pair \citep{Kucera88zp}
\item every $\PA$ degree bounds a minimal pair \citep{jocksoarepimin}.
\end{enumerate}
Toward a strengthening of the join property with randoms shown in Theorem \ref{C1JiRsNKu}, one would consider the property that
reals in $\PA$ bound a minimal pair of randoms (before looking at being the join of a minimal pair of randoms).
However even this basic property fails  due to (i) and the fact that randoms are $\DNC$.
This failure contrasts with (ii).

\section{Random join property for PA degrees}\label{pF9WGGxG8R}
We prove Theorem \ref{C1JiRsNKu}:
if $z$ is $\PA$ and $x<_T z$ is random, there exists random $y<_T z$ such that $x\oplus y \equiv_T z$.
To this end, we need to code $z$ into $y$, in a way that $z$ is recoverable from $y$ with the help of oracle $x$.

\subsection{Coding into a random}\label{NgRyVYgtoa}
The coding we use is based on the standard coding method into randoms by \citet{gacsMR859105} and \citet{MR820784}.
The general form of this method is reviewed by \citet{codico} along with its limitations, some of which are relevant to our task.
The idea is to start with a \pz class of positive measure which contains only randoms.
A \pz class of reals is the set of paths $[P] \subseteq \twome$ of a \emph{\pz tree} $P$, 
 namely a \pz downward prefix-closed subset of $\twomel$, without dead-ends (every node in $P$ has a successor in $P$).
\citet{MR820784} showed that  
each $\tau\in P$ has a computable positive lower-bound on $\mu([P]\cap \dbra{\tau})$, where
$\dbra{\tau}:=\sqbrad{x\in[P]}{\tau\prec x}$. This implies that
there exists computable increasing $(\ell_n)$ such that
\begin{equation}\label{GuqLmDSoQO}
\parbox{9cm}{ each $\tau\in P\cap 2^{\ell_n}$ has at least two extensions in $\tau\in P\cap 2^{\ell_{n+1}}$.}
\end{equation}
\citet{codico} noted  that if $(\ell_n)$ is computable and increasing:
\begin{equation*}
\parbox{10cm}{\eqref{GuqLmDSoQO} holds for all sufficiently large $n$ iff $\sum_n 2^{-(\ell_{n+1}-\ell_n)}<\infty$.}
\end{equation*}
This method was used in the \KG theorem: every real $z$ is computable by a random $y$. The coding consists of choosing
$y\restr_{\ell_{n+1}}$ as the leftmost or rightmost extension of $y\restr_{\ell_n}$ in $P$, according to whether $z(n)$ is 0 or 1.
The decoding of $y$ into $z$ is possible due to the fact that $P$ has a computable monotone approximation.

Our method is based on this type of coding, except that we cannot let $y$ be locally on the 
boundary of $P$:\footnote{in the sense that $y\restr_{\ell_{n+1}}$ is the leftmost or rightmost $\ell_{n+1}$-bit extension of $y\restr_{\ell_{n}}$ in $P$}  
such reals are known to be Turing complete. 
Let $\mu_{\sigma}([P]):=2^{|\sigma|}\cdot \mu([P]\cap \dbra{\sigma})$.
\citet[Corollary 2.10]{codico} observed that if
\begin{equation}\label{TwurKAqQp}
\sum_{i} q_i<\infty\wedga \forall n,\ \ell_{n+1}>-\log q_n+\ell_n
\end{equation}
for positive computable rational and integer sequences $(q_n)$, $(\ell_n)$,  then
\begin{equation}\label{ZDuCOjbw6T}
\exists k_0\ \forall n\geq k_0\ \forall \sigma\in 2^{\ell_n}\cap P: \mu_{\sigma}([P])>q_n.
\end{equation}
For the remaining part of our article let $q_n:=1/(n+1)^2$ ,
fix $P$ to be a nonempty \pz tree without dead-ends, whose paths are all random, and 
assume that $(\ell_n)$ is computable  so
by  \eqref{ZDuCOjbw6T}: 
\begin{equation}\label{U3J23QRG5a}
\textrm{each $\tau\in P\cap 2^{\ell_n}$ has at least two extensions in $\tau\in P\cap 2^{\ell_{n+1}}$.}
\end{equation}
as long as $(\ell_n)$ satisfies the latter part of \eqref{TwurKAqQp}, \eg  $\ell_n=2^n$.

\subsection{Outline of the proof}\label{cLL5rWHAZ}
Given  $z\in \PA$ and random $x<_T z$ we will find $y<_T z$ inside $[P]$ such that $x\oplus y\equiv_T z$.

We will use property \eqref{zEnN2ZYwr6} of the $\PA$ reals $z$: 
\begin{equation}\label{RgCaBDZUpT}
\parbox{10cm}{given a nonempty \pz set $Q\subseteq\twomel$ we can 
$z$-effectively choose a member of $Q$, uniformly in any \pz index of $Q$.}
\end{equation}
Let $\ell=(\ell_n)$ be exponentially growing (\eg $\ell_0=0$, $\ell_n=2^n$, $n>0$)  and let
\[
[\tau]:=\sqbrad{\tau'}{\tau\preceq\tau'}.
\]
Using \eqref{RgCaBDZUpT} we can define initial segments $y\restr_{\ell_n}$ so that $y\leq_T z$ and $y\in [P]$.

A naive approach to this task is: 
at step $n$, when we choose the extension of $y\restr_{\ell_n}$
in $P\cap 2^{\ell_{n+1}}$ to choose one with even/odd number of 1s,  in order to encode $z(n)=0,1$ respectively (or, in general, use some
{\em effective partition} $F_{0}, F_{1}$  of $[y\restr_{\ell_n}]\cap  2^{\ell_{n+1}}$). 

The problem here is that at some step $n$, 
\begin{equation}\label{EZ7HNWGy2l}
\parbox{11cm}{the set of extensions of $y\restr_{\ell_n}$ in 
$P\cap 2^{\ell_{n+1}}$ may not contain representatives of both classes in the partition: 
$F_0\cap P=\emptyset\ \vee\ F_1\cap P=\emptyset$}
\end{equation}
which would force us to encode wrong information about $z$ in $y$.
This could happen since $P$ contains  random $y\geq_T \emptyset'$ and, by \citet[Theorem 3.2]{BienvenuHMN14},
$\inf_{\sigma\prec z} \mu_{\sigma}(P)=0$; in fact,  the proof of this characterization of Turing-hard random reals
can be easily extended in order to show that the density $\mu_{\sigma}(P)$ along $y\geq_T\emptyset'$ drops at the expense of any
effective partition $F_{0}, F_{1}$, making $F_i\cap P=\emptyset$ for some $i<2$. 

Hence the partition of the extensions of $y\restr_{\ell_n}$  in $P\cap 2^{\ell_{n+1}}$
cannot be fully effective, and this is where the availability of $x$ plays a role: 
it suffices to show that every random $x$ uniformly defines a partition function, such that
\eqref{EZ7HNWGy2l} does not happen, for sufficiently large $n$.
We do this in \S\ref{fzEp2Lp1eI}: after uniformly encoding all such partition systems into reals $x$ in \S\ref{lWAmAPDqRz},  
we show that the effective probability of the 
partition systems satisfying \eqref{EZ7HNWGy2l} tends to 0 as $n$ tends to infinity. In other words, if $x$ is random, 
for sufficiently large $n$ the partition system defined by $x$ fails \eqref{EZ7HNWGy2l}.
The coding of $z$ into $y$ can then be done as in the naive construction, but with the use of the partitions defined by $x$.
By the construction, $x\oplus y\equiv_T z$.

\subsection{Partition systems}\label{lWAmAPDqRz}
Let $\ell=(\ell_n)$ be computable and increasing with $\ell_0=0$, and  $m_n:=\ell_{n+1}-\ell_n$.
An {\em $\ell$-partition system} is a perfect tree whose $n$-level nodes $D_{\sigma}, \sigma\in 2^n$ are pairwise disjoint
subsets of $2^{\ell_n}$ of equal size, whose union is $2^{\ell_n}$ and which are monotone with respect to $\preceq$.
You can think of an $\ell$-partition system as an interative 
process which partitions the $\ell_{n+1}$-bit extensions of each $\tau\in 2^{\ell_n}$
into two equal-sized sets. 
Given $\tau\in 2^{\ell_n}$ there exists a unique $\sigma\in 2^n$ such that $\tau\in D_{\sigma}$ and which gives
the partition of $[\tau]\cap 2^{\ell_{n+1}}$ into
$[\tau]\cap D_{\sigma\ast 0}$, $[\tau]\cap D_{\sigma\ast 1}$.
Formally:
\begin{defi}
An {\em $\ell$-partition system} is a map $\sigma\mapsto D_{\sigma}\subseteq 2^{\ell_{n}}$, where $n = |\sigma|$, such that:
\begin{itemize}
\item $D_{\sigma\ast 0}\cap D_{\sigma\ast 1}=\emptyset$
and each string in $D_{\sigma\ast 0}, D_{\sigma\ast 1}$ has a prefix in $D_{\sigma}$,
\item $\tau\in D_{\sigma}\ \Rightarrow\ [\tau]\cap \parb{D_{\sigma\ast 0}\cup D_{\sigma\ast 1}}= [\tau]\cap 2^{\ell_{n+1}}$,
\item $\tau\in D_{\sigma}\ \Rightarrow\ \abs{[\tau]\cap D_{\sigma\ast 0}}=\abs{[\tau]\cap D_{\sigma\ast 1}}= 2^{m_n-1}$.
\end{itemize}
For simplicity, given a partition-system $\sigma\mapsto D_{\sigma}$ we:
\begin{itemize}
\item call the restriction of $\sigma\mapsto D_{\sigma}$ to $2^{\leq n}$ a {\em partition-system of height $n$}
\item fix $\ell_n:=2^n$ and refer simply to {\em partition-systems}\footnote{$\ell_n$ is the length of the prefix of $x$ needed for  the computation of $z$ from $x,y$; later we will see that setting $\ell_n=n\cdot \ceil{\log n}$ suffices for our result.}
\end{itemize}
Let $\BB$ be the class of partition-systems and
$\BB_n$ the class of partition-systems of height $n$.
\end{defi}
For each $\tau\in 2^{\ell_n}$ there exists a unique $\sigma\in 2^n$ such that $\tau\in D_{\sigma}$.

Note that $D_{\sigma}, \sigma\in 2^n$ is itself a partition of $2^{\ell_n}$ into $2^n$ parts.
For $\sigma\in 2^n, \tau\in D_{\sigma}$ we are interested in 
the partition of $[\tau]\cap 2^{\ell_{n+1}}$ into $D_{\sigma\ast 0}\cap [\tau], D_{\sigma\ast 1}\cap [\tau]$.
So $(D_{\sigma\ast 0}, D_{\sigma\ast 1})$ can be viewed as the group of partitions 
$(D_{\sigma\ast 0}\cap [\tau], D_{\sigma\ast 1}\cap [\tau])$, $\tau\in D_{\sigma}$.\footnote{This is why in many expressions we need to include $[\tau]$: to specify that we refer to the split of the $\ell_{n+1}$-bit extensions of $\tau$ that
the partition-system defines.}
\begin{defi}[Codes for partitions]
Fix a computable increasing $(u_n)$ 
and an effective naming $f$ of partition-systems of finite height by strings such that:
\begin{itemize}
\item each $\sigma\in 2^{u_n}$ corresponds to a unique partition-system $f(\sigma)\in \BB_n$
\item each extension of $\sigma$ in $2^{u_{n+1}}$ corresponds to an  extension of $f(\sigma)$ in $\BB_{n+1}$.
\end{itemize}
We get an effective naming of $\BB$ by reals: the computable surjection $f:\twome\to \BB$ given by
$f(x):=\lim_n f(x\restr_{u_n})$.
Let 
\begin{itemize}
\item $\sigma\mapsto D^x_{\sigma}$ denote the partition-system $f(x)$
\item $\sigma\mapsto D^{\eta}_{\sigma}$ denote the partition-system $f(\eta)$, when $\eta\in 2^{u_n}$ for some $n$.
\end{itemize}
\end{defi}
Each real computes the height-$n$  prefix of its corresponding member of $\BB$ with computable oracle-use $u_n$.
Recall $q_n:=1/(n+1)^2$ from \S\ref{NgRyVYgtoa} and that $\ell_n:=2^n$, $m_n:=\ell_{n+1}-\ell_n$. Hence: 
\begin{equation}\label{1hlkq7rPp6}
2^{\ell_n+1}\cdot 2^{-q_n^2\cdot 2^{m_n}}< 2^{-n}
\hspace{0.3cm}\textrm{for sufficiently large $n$}
\end{equation}
which we will use along with \eqref{U3J23QRG5a}.\footnote{In the end of \S\ref{fzEp2Lp1eI} we will see that setting $\ell_n=n\cdot \ceil{\log n}$ suffices for our results.}

{\em Remark.} For a later discussion on the oracle-uses of our computations it is interesting to estimate the size of $u_n$, which can be done by
estimating $\abs{\BB_n}$: 
\begin{equation*}
\parbox{11cm}{for $\ell_n=2^n$ we get $u_n\leq 2^{2^{2^n}}$ and for $\ell_n=n\cdot \log n$ we get $u_n\leq 2^{2^{n^2}}$}
\end{equation*}
which are straightforward, and overwhelming compared to $\ell_n$.

\subsection{Counting the failed partition systems}\label{fzEp2Lp1eI}
Recall  that $[\tau]$ denotes the set of extensions of $\tau$ in $\twomel$.
We formalize the failure discussed in \S\ref{cLL5rWHAZ}, in terms of condition 
\eqref{EZ7HNWGy2l} that we try to avoid in partition systems.
\begin{defi}[Failure]
Given  $\tau\in 2^{\ell_n}$, $\eta\in 2^{u_{n+1}}$, $x\succ\eta$  and  the corresponding $(D^{\eta}_{\sigma})$:
\begin{itemize}
\item let $\sigma^{\eta}_{\tau}$ denote the unique $\sigma\in 2^n$  
such that $\tau\in D^{\eta}_{\sigma}$ and $\sigma^{x}_{\tau}:=\sigma^{\eta}_{\tau}$
\item we simplify the expression $D^{\eta}_{\sigma^{\eta}_{\tau}}$ into $D^{\eta}_{\sigma_{\tau}}$
and $D^{x}_{\sigma^{x}_{\tau}}$ into $D^{x}_{\sigma_{\tau}}$
\item say that 
system $(D^{\eta}_{\sigma})$ {\em fails at $\tau$} if 
$[\tau]\cap D^{\eta}_{\sigma_{\tau}\ast 0}\cap P=\emptyset \vee [\tau]\cap D^{\eta}_{\sigma_{\tau}\ast 1}\cap P=\emptyset$.
\end{itemize}
Given $x\succ\eta$ we say that $(D^{x}_{\sigma})$ {\em fails at $\tau$} if $(D^{\eta}_{\sigma})$  fails at $\tau$. 
\end{defi}
We stress that $D^{\eta}_{\sigma_{\tau}}$ is not necessarily a subset of $D^{\eta}_{\sigma_{\tau}}\cap [\tau]$: 
distinct $\tau,\tau'$ may have $\sigma^{\eta}_{\tau}=\sigma^{\eta}_{\tau'}$, hence
$D^{\eta}_{\sigma_{\tau}}=D^{\eta}_{\sigma_{\tau'}}$ and $D^{\eta}_{\sigma_{\tau}\ast i}, i<2$ contains the partitions of both
$[\tau]\cap 2^{\ell_{n+1}}$ and $[\tau']\cap 2^{\ell_{n+1}}$.

\begin{defi}[Fail-sets]
The set partition-systems that fail at $\tau\in 2^{\ell_n}$ is $G^0_{\tau}\cup G^1_{\tau}$ where: 
\[
G^i_{\tau}:=\sqbrad{\eta\in 2^{u_{n+1}}}{[\tau] \cap P \cap D^{\eta}_{\sigma_{\tau}\ast i} = \emptyset}
\hspace{0.2cm}\textrm{so}\hspace{0.2cm}
\dbra{G^i_{\tau}}=\sqbrad{x}{[\tau] \cap P \cap D^x_{\sigma_{\tau}\ast i} = \emptyset}.
\]
for $i<2$, and where we identify partition-systems  with their codes.
\end{defi}
By compactness, the $G^i_{\tau}$ are \ce uniformly in $i,\tau$. 

We  bound the probability of failure at $\tau$.
Recall the $P$-density bound $q_n$ from \S\ref{cLL5rWHAZ} and its property \eqref{U3J23QRG5a}: 
for sufficiently large $n$, each $\tau\in 2^{\ell_n}$ has $\geq q_n\cdot 2^{m_n}$ successors in $2^{\ell_{n+1}}\cap P$.
\begin{lem}\label{HLihTBIfHb}
For sufficiently large $n$, if $\tau\in 2^{\ell_n}\cap P, i<2$ then $\mu(G^i_{\tau})\leq 2^{-q_n^2\cdot  2^{m_n}}$.
\end{lem}\begin{proof}
The failure of a partition-system $(D^x_{\sigma})$  at $\tau$ depends only on $\eta:=x\restr_{u_{n+1}}$, and in particular on whether this partition
splits the extensions of $\tau$ in $P\cap 2^{\ell_{n+1}}$  into two nonempty parts. 
This is precisely modeled by the hypergeometric distribution: randomly drawing a sample of $k=2^{m_n-1}$ 
balls/strings from a box containing $N=2^{m_n}$ balls/strings, of which 
$R\geq q_n\cdot 2^{m_n}$ are black (the ones  inside $P$, and are the {\em successful draws}). 

\citet{Hoeffdinghyp} showed that in such a sample
\begin{equation*}
\parbox{10cm}{the probability that no successful draws occur is $\leq e^{-2(R/N)^2 k}$.}
\end{equation*}
Since $R/N\geq q_n$, the proportion of strings in $2^{u_{n+1}}$ that belong to $G^i_{\tau}$ has the same bound:
\[
\abs{G^i_{\tau}}\leq 2^{u_{n+1}}\cdot e^{-2q_n^2\cdot  2^{m_n-1}}
\hspace{0.3cm}\Rightarrow\hspace{0.3cm}
\mu(G^i_{\tau})< 2^{-q_n^2\cdot  2^{m_n}}.
\]
\end{proof}
An alternative route toward Lemma \ref{HLihTBIfHb} is to use the explicit expression:
\[
{N-R\choose k}\ \Big/\ {N\choose k}
\]
for the probability that a random sample of size $k$ from a population of size $N$ with $R$ successful members, contains no successful member.
A third way is to consider the probability that $k$  independent trials with probability of success
remaining above $R/N$ all fail.
\begin{lem}\label{lem:main}
If $x$ is random, there exists $n_0$ such that
\[
\parB{\sigma\in 2^{\geq n_0}\wedga \tau\in P\cap D^x_{\sigma}}
\ \ \Rightarrow\ \ 
\parB{[\tau]\cap P\cap D^x_{\sigma\ast 0}\neq\emptyset\wedga 
[\tau]\cap P\cap D^x_{\sigma\ast 1}\neq\emptyset}.
\]
\end{lem}\begin{proof}
By compactness, the $G^i_{\tau}$ are \ce uniformly in $i,\tau$, and by
Lemma \ref{HLihTBIfHb}:
\begin{equation}\label{8p57JRsuDg}
\tau\in 2^{\ell_n}\cap P
\hspace{0.3cm}\Rightarrow\hspace{0.3cm}
\mu(G^i_{\tau})< 2^{-q_n^2\cdot 2^{m_n}}
\end{equation}
for sufficiently large $n$.
We are almost ready to define the required \ml test for the proof of the lemma, except that the bound in $\eqref{8p57JRsuDg}$ is conditional on 
$\tau\in 2^{\ell_n}\cap P$. Fortunately, the latter is a \pz condition, so we may effectively restrict $G^0_{\tau}, G^1_{\tau}$ into
\ce sets $L_{\tau}, R_{\tau}$ which meet the bound of $\eqref{8p57JRsuDg}$ and {\em if} $\tau\in 2^{\ell_n}\cap P$ then they equal  
$G^0_{\tau}, G^1_{\tau}$ respectively.

Considering the effective enumerations $G^0_{\tau}(s), G^1_{\tau}(s)$, of $G^0_{\tau}, G^1_{\tau}$ let
$G^0_{\tau}(\infty):=G^0_{\tau}$ and let:
\begin{itemize}
\item   $t$ be the  least such that $\mu(G^0_{\tau}(t+1))\geq 2^{-q_n^2\cdot 2^{m_n}}$; $t:=\infty$ if such stage does not exist
\item $L_{\tau}:=G^0_{\tau}(t)$, and define $R_{\tau}$ similarly, with respect to $G^1_{\tau}$.
\end{itemize}
Letting $G_{\tau}:=L_{\tau}\cup R_{\tau}$, the  sets 
$L_{\tau}, R_{\tau}, G_{\tau}$ are uniformly \ce and by \eqref{8p57JRsuDg}:
\begin{enumerate}[\hspace{0.3cm}(a)]
\item $\tau\in 2^{\ell_n}\cap P\ \Rightarrow\  \parb{L_{\tau}=G^0_{\tau}\wedga R_{\tau}=G^1_{\tau}}$, for sufficiently large $n$
\item $\tau\in 2^{\ell_n}\ \Rightarrow\ 
\mu(G_{\tau})\leq \mu(L_{\tau})+\mu(R_{\tau})<  2\cdot 2^{-q_n^2\cdot 2^{m_n}}$.
\end{enumerate}
Hence $\mu(G_{\tau})< 2\cdot 2^{-q_n^2\cdot 2^{m_n}}$ for all $\tau\in 2^{\ell_n}$ and all $n$. Letting 
$G_n:=\cup_{\tau\in 2^{\ell_n}} G_{\tau}$ we get:
\[
\mu(G_{n})\leq  \sum_{\tau\in 2^{\ell_n}} \mu(G_{\tau})\leq
2^{\ell_n+1}\cdot 2^{-q_n^2\cdot 2^{m_n}}< 2^{-n} 
\]
where we used \eqref{1hlkq7rPp6}, and the $G_{n}$ are \ce uniformly in $n$.

Hence $(G_{n})$ is a \ml test
and if $x$ is random, $\exists n_0 \ \forall n\geq n_0: x\not\in G_n$.

Finally, by clause (a) above and the discussion in the first part of this section, 
for sufficiently large $n$ the set $G_n$ contains the partition-systems that fail at some 
$\tau\in P\cap 2^{\ell_n}$. Hence if $x$ is random, it will stop failing from some level $n_0$ on, which completes the proof.
\end{proof}

{\em Remark.} 
The choice $\ell_n=2^n$ that we made above \eqref{1hlkq7rPp6} is far from optimal, and
$\ell_n=n\ceil{\log n}$ suffices for Lemma \ref{lem:main}. Indeed, it suffices to ensure that $\sum_n \mu(G_{n})<\infty$,
so it suffices that $\mu(G_{n})\leq 1/n^2$. 
By \eqref{8p57JRsuDg}, it suffices that $m_n:=\ell_{n+1}-\ell_n>5\log n$. Since 
\[
\sum_{i<n} 5 \log i \leq 5 \log n! < n\log n
\]
for sufficiently large $n$, it suffices to take $\ell_n=n\cdot \ceil{\log n}$. 
This choice also satisfies the latter part of \eqref{TwurKAqQp}, which was a previous commitment. 

\subsection{Remaining proof for Theorem \ref{C1JiRsNKu} and Corollary \ref{bZmBhYeMn}}
We are now ready to finish the proof: let $z$ be $\PA$ and let 
$x<_T z$ be random. 

We will construct random $y<_T z$ such that $x\oplus y \equiv_T z$.
Since $x$ is random, there exists $n_0$ such that for each $\sigma\in 2^{\geq n_0}$:
\begin{equation}\label{TO6TQdJKG}
\tau\in P\cap D^x_{\sigma}
\ \ \Rightarrow\ \ 
\parB{[\tau]\cap P\cap D^x_{\sigma\ast 0}\neq\emptyset\wedga 
[\tau]\cap P\cap D^x_{\sigma\ast 1}\neq\emptyset}.
\end{equation}
We define the required $y\leq_T z$ as follows, starting from some $\sigma_0\in 2^{n_0}$
and some $\tau_0\in P\cap D^x_{\sigma_0}$. By \eqref{TO6TQdJKG} we may use $z$ to choose
$\tau_1\in [\tau_0]\cap P\cap D^x_{\sigma_0\ast z(0)}$, also letting $\sigma_1:=\sigma_0\ast  z(0)$. By the same token, we then choose 
$\tau_2\in [\tau_1]\cap P\cap D^x_{\sigma_1\ast z(1)}$,
also letting $\sigma_2:=\sigma_1\ast  z(1)$
and so on. 

Then $(\tau_i)$ is well-defined and $z$-computable since
\begin{equation*}
\parbox{9cm}{for $j<2$ the sets $[\tau_i]\cap P\cap D^x_{\sigma\ast j}$ are nonempty and  \pz }
\end{equation*}
and each can be obtained uniformly using a finite prefix of $x$.

Let $y:=\cup_n \tau_n$ and it remains to show that $z\leq_T x\oplus y$.

Starting from $\tau_0$, $\sigma_0$, we may use $x$  to compute $D_{\sigma_0\ast j}^x, j<2$, 
and let $z(0)=i$ for the unique $i$ such that $y$ has a prefix in $D_{\sigma_0\ast i}$.
In general, given $\tau_k, \sigma_k$ we
\begin{itemize}
\item use $x$  to compute $D_{\sigma_k\ast j}^x, j<2$
\item let $z(k)=i$ for the unique $i$ such that $y$ has a prefix in $D_{\sigma_k\ast i}$.
\end{itemize}
This shows that $z\leq_T x\oplus y$.

By the canonical coding of partition systems in \S\ref{lWAmAPDqRz},
$x$ truth-table computes its partition system $D^x_{\tau}$. In the above reduction
$z\leq_T x\oplus y$, the computation of $z$ is via a truth-table on 
the system $(D^x_{\tau})$ and $y$. Hence if $x<_{tt} z$, we get $z\equiv_{tt} x\oplus y$,
and a similar statement holds for weak truth-table reductions.

{\bf Proof of Corollary \ref{bZmBhYeMn}.}
Clause (a) of Corollary \ref{bZmBhYeMn} follows from Theorem \ref{C1JiRsNKu} and the 
relativization of the $\ze$-dominated basis theorem: every nonempty \pz class contains a $\ze$-dominated real.
Recall that the $\PA$ reals, viewed as diagonally non-computable binary function, form a \pz class
which is universal in the sense that every infinite path can compute some path through every \pz class, uniformly in its index \citep{MR0316227}.
This class can be relativized with respect to any real $x$, obtaining a
$\Pi^0_1(x)$ class $P^x$ with the property that all $y\in P^x$ have 
$y\geq_T x$ and for each  nonempty $Q\in \Pi^0_1(x)$ there exists $w\leq_T y$, $w\in Q$.
\begin{lem}\label{ahnFgRuKo}
If $x$ is $\ze$-dominated, there exists a $\ze$-dominated $z\geq_T x$ of $\PA$ degree.
\end{lem}\begin{proof}
By the definition $P^x$,  each $y\in P^x$ has $\PA$ degree.
By the $\ze$-dominated basis theorem relativized to $x$, there exists $x$-dominated $z\in P^x$, so every $f\leq_T z$
is dominated by some $g\leq_T x$. 
Hence $z\geq_T x$ is $\ze$-dominated and of $\PA$ degree.
\end{proof}
Any $x\in \twome$ is truth-table reducible to a $\PA$ real. Hence
clause (b) of Corollary \ref{bZmBhYeMn} follows from
the truth-table version of Theorem \ref{C1JiRsNKu} and
the fact by \citet{moserdepth} that all reals in the truth-table degree of a $\PA$ real are deep.

{\em Remark.} In the computation of $z$ from $x, y$  that we obtained,  the oracle-use of $x$ is $u_n$ (from the coding of partitions systems into reals that we fixed  in \S\ref{lWAmAPDqRz}) and the oracle-use of $y$ is $\ell_n$. We estimated $u_n$ to be more than double-exponential in $n$ (the length the prefix of $z$ computed)
while in the end of \S\ref{fzEp2Lp1eI} we explained that $\ell_n:=n\cdot \ceil{\log n}$ is a valid choice in our argument.

\section{Conclusion and discussion}\label{YzKdhxVSuD}
We showed that the complete extensions of  arithmetic in the degrees of unsolvability (Turing, truth-table, weak truth-table)
have the  join property with respect to random reals. In \S\ref{pY7jSdtn81} we explained that known results do not allow
any obvious strengthening of this result, in terms of complementation or even a meet property with randoms. 

Regarding the reductions of random reals and reals in $\PA$, we would like to know how tight can the 
relevant oracle-use be: what is the
\begin{enumerate}[\hspace{0.3cm}(i)]
\item  oracle-use required  in a reduction of a random real to a $\PA$ real? 
\item  oracle-use in Turing equivalence of a $\PA$ real with the join of two randoms?
\end{enumerate}
As discussed in \S\ref{iDz5who56}, these are truth-table reductions, but we care about the length of the prefix of the oracle
needed for the $n$-bit prefix of the output.
Regarding (i), we know it cannot be $n+\bigon$: \citet{ranlinMR2286414} showed that if $x$ is random and $x\leq_T y$ with
oracle-use $n+\bigon$ then $y\leq_T x$; this cannot happen when $x$ is incomplete since, as discussed in \S\ref{iDz5who56}, 
incomplete randoms do not compute $\PA$ reals.

With regard to (ii), the optimal oracle-use in computations of reals by randoms was characterized by  
\citet{optcod}. Some oracle-uses in our reductions between $\PA$ reals and joins of randoms are double-exponential (recall the explicit bounds in \S\ref{lWAmAPDqRz}, \S\ref{fzEp2Lp1eI}).\footnote{In our argument, the large oracle-use is due to the coding of partition-systems by reals.}
In the weaker result by \citet{minpal} discussed in \S\ref{iDz5who56}, the oracle-use is super-exponential. 
These sizes may be unavoidable, and a worthwhile task would be to establish such a characterization.

Another question stemming from \citet{MR2258713frank} who showed that incomplete randoms are computationally far from
$\PA$ reals, is
\begin{equation*}
\parbox{9cm}{how well can an incomplete random trace a $\PA$ real? }
\end{equation*}
A common notion of tracing is {\em c.e.-tracing}: a function $f$ is {\em $x$-\ce traceable} 
if there exist uniformly $x$-\ce $(D_i)$, $\abs{D_i}\leq n$ such that $\forall n,\ f(n)\in D_n$.
Some facts about traceability of $\PA$ and $\DNC$ reals were established by
\citet{KjosHanssenMStrans}.

A dual aspect of this question concerns the fact that incomplete randoms compute $\DNC$ functions but not
$\DNC$ functions with fixed finite range. The exact condition for the growth of the range of a $\DNC$ function which is computable
by an incomplete random was given by Miller (see  \citet[\S 1.2]{BienvenuHMN14}): given a computable $f$, 
an incomplete random computes an $f$-bounded $\DNC$ function  iff $\sum_n 1/f(n)<\infty$.

\bibliographystyle{abbrvnat}
\bibliography{randim}

\begin{thebibliography}{26}
\providecommand{\natexlab}[1]{#1}
\providecommand{\url}[1]{\texttt{#1}}
\expandafter\ifx\csname urlstyle\endcsname\relax
  \providecommand{\doi}[1]{doi: #1}\else
  \providecommand{\doi}{doi: \begingroup \urlstyle{rm}\Url}\fi

\bibitem[Barmpalias and Lewis(2007)]{ranlinMR2286414}
G.~Barmpalias and A.~E.~M. Lewis.
\newblock Randomness and the linear degrees of computability.
\newblock \emph{Ann. Pure Appl. Logic}, 145\penalty0 (3):\penalty0 252--257,
  2007.

\bibitem[Barmpalias and Lewis-Pye(2015)]{typvol}
G.~Barmpalias and A.~Lewis-Pye.
\newblock The information content of typical reals.
\newblock In \emph{Turing's revolution}, pages 207--224. Birkh\"auser/Springer,
  Cham, 2015.

\bibitem[Barmpalias and Lewis-Pye(2016)]{optcod}
G.~Barmpalias and A.~Lewis-Pye.
\newblock Optimal redundancy in computations from random oracles.
\newblock \emph{J. Comput. System Sci.}, 82:\penalty0 1283--1299, 2016.

\bibitem[Barmpalias and Lewis-Pye(2020)]{codico}
G.~Barmpalias and A.~Lewis-Pye.
\newblock Limits of the {K}u{\v{c}}era-{G}{\'a}cs coding method.
\newblock In \emph{Structure and Randomness in Computability and Set Theory
  (edited with Douglas Cenzer, Chris Porter and Jindrich Zapletal)}, pages
  87--109. World Scientific Press, 2020.

\bibitem[Barmpalias et~al.(2010)Barmpalias, Lewis, and Ng]{minpal}
G.~Barmpalias, A.~E.~M. Lewis, and K.~M. Ng.
\newblock The importance of \pz classes in effective randomness.
\newblock \emph{J.\ Symb.\ Log.}, 75\penalty0 (1):\penalty0 387--400, 2010.

\bibitem[Barmpalias et~al.(2014)Barmpalias, Day, and Lewis-Pye]{typical}
G.~Barmpalias, A.~R. Day, and A.~E.~M. Lewis-Pye.
\newblock The typical {T}uring degree.
\newblock \emph{Proc. Lond. Math. Soc. (3)}, 109\penalty0 (1):\penalty0 1--39,
  2014.

\bibitem[Bennett(1988)]{Bennett1988}
C.~H. Bennett.
\newblock Logical depth and physical complexity.
\newblock In R.~Herken, editor, \emph{The universal {Turing machine}, a half
  century survey}, pages 227--257. Oxford U.P., 1988.

\bibitem[Bienvenu et~al.(2014)Bienvenu, H{\"{o}}lzl, Miller, and
  Nies]{BienvenuHMN14}
L.~Bienvenu, R.~H{\"{o}}lzl, J.~S. Miller, and A.~Nies.
\newblock {D}enjoy, {D}emuth, and density.
\newblock \emph{J. Math. Logic}, 14, 2014.

\bibitem[Diamondstone et~al.(2010)Diamondstone, Dzhafarov, and
  Soare]{DzhafarovPA}
D.~E. Diamondstone, D.~D. Dzhafarov, and R.~I. Soare.
\newblock {$\Pi^0_1$ Classes, Peano Arithmetic, Randomness, and Computable
  Domination}.
\newblock \emph{Notre Dame J.\ Formal Logic}, 51\penalty0 (1):\penalty0
  127--159, 2010.

\bibitem[Franklin et~al.(2011)Franklin, Stephan, and Yu]{relafrankjoh}
J.~N.~Y. Franklin, F.~Stephan, and L.~Yu.
\newblock Relativizations of randomness and genericity notions.
\newblock \emph{Bull. Lond. Math. Soc.}, 43:\penalty0 721--733, 2011.

\bibitem[G{\'a}cs(1986)]{gacsMR859105}
P.~G{\'a}cs.
\newblock Every sequence is reducible to a random one.
\newblock \emph{Inform.\ and Control}, 70\penalty0 (2-3):\penalty0 186--192,
  1986.

\bibitem[G\"{o}del(1931)]{godel31}
K.~G\"{o}del.
\newblock {\"{U}ber formal unentscheidbare S\"{a}tze der Principia Mathematica
  und verwandter Systeme I}.
\newblock \emph{Monash. Math. Phys.}, 38\penalty0 (1):\penalty0 173--198, 1931.

\bibitem[Greenberg et~al.(2021)Greenberg, Miller, and Nies]{pamiller}
N.~Greenberg, J.~S. Miller, and A.~Nies.
\newblock Highness properties close to {PA} completeness.
\newblock \emph{Isr. J. Math.}, 244:\penalty0 419--465, 2021.
\newblock Arxiv:1912.03016.

\bibitem[Hoeffding(1963)]{Hoeffdinghyp}
W.~Hoeffding.
\newblock Probability inequalities for sums of bounded random variables.
\newblock \emph{J.~Am.~Stat.~Assoc.}, 58\penalty0 (301):\penalty0 13--30, 1963.

\bibitem[Jockusch and Soare(1971)]{jocksoarepimin}
C.~G. Jockusch and R.~I. Soare.
\newblock A minimal pair of {$\Pi^0_1$} classes.
\newblock \emph{J.\ Symb.\ Log.}, 36\penalty0 (1):\penalty0 66--78, 1971.

\bibitem[Jockusch and Soare(1972)]{MR0316227}
C.~G. Jockusch, Jr. and R.~I. Soare.
\newblock {$\Pi \sp{0}\sb{1}$} classes and degrees of theories.
\newblock \emph{Trans. Amer. Math. Soc.}, 173:\penalty0 33--56, 1972.

\bibitem[Kjos-Hanssen et~al.(2011)Kjos-Hanssen, Merkle, and
  Stephan]{KjosHanssenMStrans}
B.~Kjos-Hanssen, W.~Merkle, and F.~Stephan.
\newblock Kolmogorov complexity and the recursion theorem.
\newblock \emph{Trans. Amer. Math. Soc.}, 363, 2011.

\bibitem[Ku{\v{c}}era(1985)]{MR820784}
A.~Ku{\v{c}}era.
\newblock Measure, {$\Pi\sp 0\sb 1$}-classes and complete extensions of {${\rm
  PA}$}.
\newblock In \emph{Recursion theory week (Oberwolfach, 1984)}, volume 1141 of
  \emph{Lecture Notes in Math.}, pages 245--259. Springer, Berlin, 1985.

\bibitem[Ku{\v{c}}era(1988)]{Kucera88zp}
A.~Ku{\v{c}}era.
\newblock On the role of $0'$ in recursion theory.
\newblock In \emph{Logic Colloquium '86 (Hull, 1986)}, volume 124 of
  \emph{Stud. Logic Found. Math.}, pages 133--141. North-Holland, Amsterdam,
  1988.

\bibitem[Levin(2013)]{Levin2013FI}
L.~A. Levin.
\newblock Forbidden information.
\newblock \emph{J. ACM}, 60\penalty0 (2):\penalty0 9:1--9:9, 2013.

\bibitem[Lewis-Pye(2012)]{lewisjoinnote}
A.~Lewis-Pye.
\newblock A note on the join property.
\newblock \emph{Proc.\ Amer.\ Math.\ Soc.}, 140\penalty0 (2):\penalty0
  707--714, 2012.

\bibitem[Martin-L{\"o}f(1966)]{MR0223179}
P.~Martin-L{\"o}f.
\newblock The definition of random sequences.
\newblock \emph{Inform.\ and Control}, 9:\penalty0 602--619, 1966.

\bibitem[Moser(2013)]{MOSER201396}
P.~Moser.
\newblock On the polynomial depth of various sets of random strings.
\newblock \emph{Theoret.\ Comput.\ Sci.}, 477:\penalty0 96--108, 2013.

\bibitem[Moser and Stephan(2015)]{moserdepth}
P.~Moser and F.~Stephan.
\newblock Depth, highness and {DNR} degrees.
\newblock In A.~Kosowski and I.~Walukiewicz, editors, \emph{Fundamentals of
  Computation Theory}, pages 81--94, Cham, 2015. Springer International
  Publishing.

\bibitem[Scott and Tennenbaum(1960)]{ScottTennen}
D.~S. Scott and S.~Tennenbaum.
\newblock On the degrees of complete extensions of arithmetic.
\newblock \emph{Notices Am.\ Math.\ Soc.}, 7:\penalty0 242--243, 1960.

\bibitem[Stephan(2006)]{MR2258713frank}
F.~Stephan.
\newblock Martin-{L}\"{o}f random and {${\rm PA}$}-complete sets.
\newblock In \emph{Logic Colloquium '02}, volume~27 of \emph{Lect. Notes Log.},
  pages 342--348. Assoc. Symbol. Logic, La Jolla, CA, 2006.

\end{thebibliography}

\end{document}